\documentclass[11pt]{amsart}

\usepackage{verbatim}
\usepackage{amssymb, amsmath, amsthm}
\usepackage{hyperref}
\usepackage[alphabetic,lite]{amsrefs}
\usepackage{tikz}
\usetikzlibrary{matrix,arrows,decorations.pathmorphing}
\usepackage[all]{xy}
\usepackage{tikz-cd}

\newtheorem{lemma}{Lemma}[section]
\newtheorem{theorem}[lemma]{Theorem}

\newtheorem{prop}[lemma]{Proposition}
\newtheorem{cor}[lemma]{Corollary}

\newtheorem{claim*}{Claim}
\newtheorem{thm}[lemma]{Theorem}
\newtheorem*{thm*}{Theorem}

\theoremstyle{definition}
\newtheorem{defn}[lemma]{Definition}

\newtheorem{conj}[lemma]{Conjecture}

\newtheorem{question}[lemma]{Question}

\theoremstyle{remark}
\newtheorem{rmk}[lemma]{Remark}

\DeclareMathOperator{\Kum}{Kum}

\newcommand{\A}{{\bf A}}
\newcommand{\G}{{\bf G}}

\newcommand{\PP}{{\bf P}}

\newcommand{\F}{{\bf F}}
\newcommand{\Q}{{\bf Q}}

\newcommand{\Z}{{\bf Z}}

\newcommand{\Xbar}{{\overline{X}}}

\newcommand{\kbar}{{\overline{k}}}

\newcommand{\Ybar}{{\overline{Y}}}

\newcommand{\calM}{{\mathcal M}}

\newcommand{\calO}{{\mathcal O}}

\newcommand{\calV}{{\mathcal V}}

\newcommand{\calX}{{\mathcal X}}

\newcommand{\frakF}{{\mathfrak F}}

\usepackage{mathrsfs} 
\newcommand{\scrA}{{\mathscr A}}

\newcommand{\scrK}{{\mathscr K}}

\newcommand{\scrR}{{\mathscr R}}

\DeclareMathOperator{\inv}{inv}

\DeclareMathOperator{\im}{im}

\DeclareMathOperator{\Hom}{Hom}

\DeclareMathOperator{\Aut}{Aut}
\DeclareMathOperator{\Gal}{{Gal}}

\DeclareMathOperator{\Cor}{Cor}
\DeclareMathOperator{\Res}{Res}

\DeclareMathOperator{\Br}{{Br}}

\DeclareMathOperator{\Pic}{Pic}

\DeclareMathOperator{\Spec}{{Spec}}

\DeclareMathOperator{\et}{\textrm{\normalfont \'et}}

\DeclareMathOperator{\N}{N}

\DeclareMathOperator{\GL}{GL}

\DeclareMathOperator{\NS}{NS}

\DeclareMathOperator{\pr}{pr}

\numberwithin{equation}{section}
\numberwithin{table}{section}

\newcommand{\defi}[1]{\textsf{#1}} 

\newcommand{\CH}{\textrm{\normalfont CH}}

\usepackage{geometry} 
\begin{document}
\title[Arithmetic of rational points and zero-cycles on products]{Arithmetic of rational points and zero-cycles on products of Kummer varieties and K3 surfaces}

\author{Francesca Balestrieri}
\address{Francesca Balestrieri\\ 
Max-Planck-Institut f{$\ddot{\textrm{u}}$}r Mathematik\\
Vivatsgasse 7\\
 53111 Bonn\\
 Germany.}
 \email{fbales@mpim-bonn.mpg.de}

\author{Rachel Newton}
\address{Rachel Newton\\
Department of Mathematics and Statistics\\ 
University of Reading\\
Whiteknights\\
PO Box 220\\
Reading RG6 6AX\\ 
UK.}
\email{r.d.newton@reading.ac.uk}

\thanks{MSC 2010: 11G35 (primary), 14G05, 14G25 (secondary). \emph{Keywords}: zero-cycle, Brauer-Manin obstruction, Hasse principle, weak approximation, Kummer variety, K3 surface.}

\begin{abstract}
Let $k$ be a number field. In the spirit of a result by Yongqi Liang, we relate the arithmetic of rational points over finite extensions of $k$ to that of zero-cycles over $k$ for Kummer varieties over $k$.  For example, for any Kummer variety $X$ over $k$, we show that if the Brauer-Manin obstruction is the only obstruction to the Hasse principle for rational points on $X$ over all finite extensions of $k$, then the ($2$-primary) Brauer-Manin obstruction is the only obstruction to the Hasse principle for zero-cycles of any given odd degree on $X$ over $k$. 
We also obtain similar results for products of Kummer varieties, K3 surfaces and rationally connected varieties.

\end{abstract}

\maketitle

\section{Introduction}

Let $X$ be a smooth, proper, geometrically integral variety over a number field $k$ with a fixed choice of algebraic closure $\kbar$, let $\overline{X}=X\times_{\Spec k} \Spec\overline{k}$,
and let $\Br X:= \mathrm{H}^2_{\et}(X, \G_m)$ be the Brauer group of $X$.  In \cite{Manin}, Manin described a pairing \mbox{\(\langle \ , \  \rangle_{\textrm{BM}} : X(\A_k) \times \Br X \to \Q/\Z\)} such that the set of rational points $X(k)$ is contained in the subset $X(\A_k)^{\Br X}$ of adelic points that are orthogonal to $\Br X$ under the pairing. It can (and sometimes does) happen that $X(\A_k)^{\Br X} = \emptyset $ while $X(\A_k) \neq \emptyset$, yielding an obstruction to the Hasse principle on $X$. This obstruction is known as the Brauer-Manin obstruction (for more details, we refer the reader to e.g. \cite{Skorobogatov-Torsors}). Variations of the Brauer-Manin obstruction can be obtained by considering subsets $B \subset \Br X$ instead of the full Brauer group. For example, one can use the algebraic Brauer group $\Br_1 X:= \ker( \Br X \to \Br \Xbar)$, or the $d$-primary torsion subgroup $\Br X \{d\}$.
Although the Brauer-Manin obstruction may not always suffice to explain the failure of the Hasse principle (see e.g. \cite{Skorobogatov-1999}, \cite{Poonen-2010}), it is nonetheless known to be sufficient for some classes of varieties and is conjectured to be sufficient for geometrically rationally connected varieties (see \cite{CT03}) and K3 surfaces (see \cite{Skorobogatov-K3Conj}). For a survey of recent progress and open problems 
in the field of rational points and obstructions to their existence, we refer the reader to \cite{Wittenberg}.

In general, it is interesting to also consider 0-cycles on varieties, rather than just rational points. Recall that a 0-cycle on $X$ is a formal sum of closed points of $X$, $z=\sum n_i P_i$, with degree $\deg(z)=\sum n_i [k(P_i):k]$. A rational point of $X$ over $k$ is thus a 0-cycle of degree $1$. Conjecturally, the Brauer-Manin obstruction to the Hasse principle for 0-cycles of degree $1$ is the only one for smooth, proper, geometrically integral varieties over $k$ (see \cite{KS86}, \cite{CT95}). In this paper,  we focus on the relationship between the arithmetic of the rational points of a variety $X$ over extension fields and the arithmetic of 0-cycles on $X$. We consider the set $\overline{\scrK}_{k,g}$ of smooth, projective, geometrically integral varieties over $k$ of dimension $g \geq 2$ which become Kummer varieties upon base change to $\overline{k}$ and its subset  $\scrK_{k,g} \subset \overline{\scrK}_{k,g}$ of Kummer varieties over $k$. We let  $  \overline{\scrK}_{k}:= \cup_{g \geq 2}    \overline{\scrK}_{k,g}$ and  $ {\scrK}_{k}:= \cup_{g \geq 2}    {\scrK}_{k,g}$.

\begin{thm}\label{thm:Kummer2-primary} Let $k$ be a number field. 
Let $W:=\prod_{i=1}^r X_i$, where  $X_i \in \scrK_k$ for all $i\in\{1,\dots,r\}$, and let $\delta$ be an odd integer. Suppose that, for all $i\in\{1,\dots,r\}$ and all finite extensions $l/k$, the Brauer-Manin obstruction is the only obstruction to the Hasse principle for rational points on $X_i$ over $l$.
 Then the $2$-primary Brauer-Manin obstruction is the only obstruction to the Hasse principle for 0-cycles of degree $\delta$ on $W$.  
\end{thm}

In \cite{Skorobogatov-K3Conj}, Skorobogatov conjectured that the Brauer-Manin obstruction is the only obstruction to the Hasse principle for rational points on K3 surfaces over number fields. This leads to the following conditional result for Kummer surfaces:

\begin{cor}\label{cor:conditional}
Let $\delta$ be an odd integer. Then
Skorobogatov's conjecture that the Brauer-Manin obstruction is the only obstruction to the Hasse principle for rational points on K3 surfaces over a number field $k$ implies that the 2-primary Brauer-Manin obstruction to the Hasse principle is the only one for 0-cycles of degree $\delta$ on products of Kummer surfaces over $k$.
\end{cor}

\begin{rmk}Note that we use the phrase `Kummer variety' in a slightly more general way than it is used in some texts. See Section~\ref{sec:Kummer} for the definition of a Kummer variety; a Kummer surface is a Kummer variety of dimension $2$. Note that Kummer surfaces are K3 surfaces (see \cite[p.2]{SZ-KumVar}).
\end{rmk}

This paper was inspired by Liang's work in \cite[Thm. 3.2.1]{Liang}, in which he takes $X$ to be a geometrically rationally connected variety over a number field $k$. He shows that if the Brauer-Manin obstruction is the only obstruction to the Hasse principle for rational points on $X$ over every finite extension of $k$, then the Brauer-Manin obstruction is the only obstruction to the Hasse principle for 0-cycles of degree 1 on $X$ over $k$.
See also \cite[Prop. 3.4.1]{Liang} for a result about 0-cycles of any degree. Liang also proved the analogue for weak approximation, in the sense defined in Definition~\ref{defn:BrWA}. Our adaptation of Liang's strategy to the case of (geometric) Kummer varieties also yields an analogous result for this notion of weak approximation:

\begin{theorem}\label{thm:main_intro_WA}
 Let $W:=\prod_{i=1}^r X_i$, where $X_i \in \overline{\scrK}_{k}$ for all $i\in\{1,\dots,r\}$, and  let $d\in\Z_{>0}$ and $\delta \in \Z$ with $\gcd(d,\delta)=1$. 
Suppose that, for all $i\in\{1,\dots,r\}$ and all finite extensions $l/k$, the $d$-primary Brauer-Manin obstruction is the only obstruction to weak approximation for rational points on $X_i$ over $l$.
 Then the $d$-primary Brauer-Manin obstruction is the only obstruction to weak approximation for 0-cycles of degree $\delta$ on $W$. 
\end{theorem}

Our key observation, which allows us to employ Liang's method, is that while it may not be possible to control the growth of the whole of $\Br X/\Br_0 X$ under an extension of the base field, in the case of a variety $X\in\overline{\scrK}_{k}$ the $d$-primary part of $\Br X/\Br_0 X$ is unchanged so long as the base field extension has degree coprime to $d$ and is linearly disjoint from a certain field extension depending on $X$. 
Control of the growth of the algebraic part $\Br_1 X/ \Br_0 X$ of the Brauer group is easily achieved. 
We remark that, since the geometric Brauer group of a geometrically rationally connected variety is finite, the growth of the whole of $\Br X/\Br_0 X$ under extensions of the base field is easily controlled in Liang's paper. In the case of Kummer varieties, however, $\Br \Xbar$ is infinite. This makes controlling the growth of the transcendental part of the Brauer group under extensions of the base field much more difficult. In order to achieve this, we use the close relationship between the Brauer group of a Kummer variety and the Brauer group of the underlying abelian variety. The exploitation of the rich structure of abelian varieties is a common theme underlying much of the recent rapid progress in the study of rational points on the related K3 surfaces and Kummer varieties (see e.g. \cite{IeronymouSkoroZarhin}, \cite{SZ-2012},  \cite{IeronymouSkoro},  \cite{Newton15},   \cite{HarpazSkoro}, \cite{VAV16},  \cite{Harpaz}, \cite{CreutzViray-DegBMO}, \cite{SZ-KumVar}). Our work in Section \ref{sec:ControlAbVar} on the transcendental part of the Brauer group of geometrically abelian varieties and geometrically Kummer varieties may be of independent interest.

After obtaining our results for Kummer varieties, we were made aware of a recent preprint of Ieronymou (\cite{Ieronymou}) in which he uses Liang's approach together with work of Orr and Skorobogatov in \cite{OrrSkoro} to prove analogues of Liang's results in \cite{Liang} for K3 surfaces. Since Kummer surfaces are both K3 surfaces and Kummer varieties of dimension $2$, there is an overlap between our work and that of Ieronymou (note that Ieronymou's results do not require that the degree of the 0-cycle is odd). We remark, however, that our results are not implied by those of Ieronymou, since we work with Kummer varieties of arbitrary dimension. Combining our results for Kummer varieties with Ieronymou's results for K3 surfaces and Liang's work on rationally connected varieties gives Theorem~\ref{thm:mixed} below, which is a generalisation of Theorem~\ref{thm:main_intro_WA} for products of these types of variety. Denote by $\overline{\scrR}_k$ the set of smooth, projective, geometrically integral, geometrically
rationally connected varieties over $k$. Denote by $\defi{K3}_k$ the set of K3 surfaces over $k$.

 \begin{theorem}\label{thm:mixed} 
Let $k$ be a number field. 
 Let $W:=\prod_{i=1}^r X_i$, where $X_i\in  \overline{\scrR}_k \cup \defi{K3}_k\cup  \overline{\scrK}_{k}$ for all $i\in\{1,\dots,r\}$. 
 Let $d\in\Z_{>0}$, let $\delta \in \Z$ and assume that $\gcd(d,\delta)=1$ if at least one of the $X_i$ is in $\overline{\scrK}_{k,g}$ with $g \geq 3$. For $X_i\in  \overline{\scrR}_k \cup \defi{K3}_k$, let $a_i=\#(\Br X_i/\Br_0 X_i)$ and let $c$ be the least common multiple of all the $a_i$ and $d$.
 Suppose that, for all finite extensions $l/k$, 
\begin{enumerate}
\item if $X_i\in  \overline{\scrR}_k \cup \defi{K3}_k$ then the Brauer-Manin obstruction is the only obstruction to the Hasse principle (respectively, to weak approximation) for rational points on $X_i$ over $l$;
\item if $X_i\in\overline{\scrK}_{k}$ then the $d$-primary Brauer-Manin obstruction is the only obstruction to the Hasse principle (respectively, to weak approximation) for rational points on $X_i$ over $l$.
\end{enumerate}

 Then the $c$-primary 
 Brauer-Manin obstruction is the only obstruction to the Hasse principle (respectively, to weak approximation)  for 0-cycles of degree $\delta$ on $W$. 
\end{theorem}

All our results thus far prove sufficiency of the Brauer-Manin obstruction for 0-cycles, asssuming the sufficiency of the Brauer-Manin obstruction for rational points over finite extensions of the base field. This leads one to ask in what circumstances it is enough to simply assume sufficiency of the Brauer-Manin obstruction to rational points over the base field. Our final results give some insight into when this may be possible.

\begin{thm}\label{thm:transfer}  Let $k$ be a number field and let $X$ be a smooth, projective, geometrically integral variety over $k$. Suppose that $X(\A_k) \neq \emptyset$ and let $[\alpha]  \in (\Br X/\Br_0 X)[2]$. 
Then there exist uncountably many finite extensions $l/k$ such that 
\[ X(\A_k)^\alpha = \emptyset \Longrightarrow  X(\A_l)^{\Res_{l/k}\alpha} = \emptyset.\]
\end{thm}

We have the following corollary for Kummer varieties:

\begin{cor}\label{thm:liftsuff}
Let $X$ be a Kummer variety over a number field $k$. Suppose that  $(\Br X/\Br_0 X)\{2\} \cong\Z/2\Z$ and that $X(\A_k) \neq \emptyset$. Then there exist uncountably many finite extensions $l/k$  such that 
\[ X(\A_k)^{\Br X} = \emptyset \Longrightarrow  X(\A_l)^{\Br X_l} = \emptyset.\]
In particular, if the Brauer-Manin obstruction to the Hasse principle  is the only one for $k$-rational points, then it is also the only one  for $l$-rational points, for all field extensions $l/k$ as above. 
\end{cor}

\subsection{Structure of the paper}
 In Section~\ref{sec:Kummer} we define the families of varieties to be studied in this paper. In Section~\ref{sec:prod} we state some preliminary results concerning Brauer groups and Brauer-Manin sets of products of varieties. In Section~\ref{sec:ControlAbVar} we show that the $d$-primary part of the transcendental Brauer group of a geometrically abelian or Kummer variety remains unchanged when the base field is extended, given certain conditions on the extension. In Section~\ref{sec:ControlKummer} we use the results of Section~\ref{sec:ControlAbVar} to show that for Kummer varieties, the $d$-primary part of the quotient $\Br X/\Br_0 X$ remains unchanged when the base field is extended, given certain conditions on the extension. In Section~\ref{sec:0-cyc} we gather various definitions concerning Brauer-Manin obstructions to local-global principles for 0-cycles. In
Section~\ref{sec:pfs} we prove our main result, which is Theorem~\ref{thm:mixed}, and deduce from it various consequences including Theorem~\ref{thm:Kummer2-primary}.
Section~\ref{sec:transfer} contains the proofs of Theorem~\ref{thm:transfer} and Corollary~\ref{thm:liftsuff}.

\subsection{Notation and terminology}   Throughout this paper, we use the following notation:  
\[
\begin{array}{rl}
k & \textrm{is a field of characteristic 0},\\
\kbar & \textrm{is a fixed algebraic closure of $k$},\\
\Gamma_k & \textrm{denotes the absolute Galois group $\Gal(\kbar/k)$},\\
\Omega_k & \textrm{is the set of all non-trivial places of $k$, when $k$ is a number field},\\
X & \textrm{is a variety over $k$},\\
X_l & \textrm{is the base-change $ X \times_{\Spec k} \Spec l$ of $X$ to a field extension $l/k$},\\
\overline{X}& \textrm{denotes $X_{\overline{k}}$},\\
\Br_1 X & \textrm{denotes $ \ker(\Br X \to \Br \Xbar)$},\\
\Br_0 X & \textrm{denotes $  \im (\Br k \to \Br X)$},\\
 \Br_1 X/\Br_0 X & \textrm{is the \emph{algebraic} part of the Brauer group of $X$},\\
  \Br X/\Br_1 X & \textrm{is the \emph{transcendental} part of the Brauer group of $X$}.\\
\end{array}\]
For an abelian group $A$ and an integer $d\in\Z_{>0}$, we use the following notation:
\[
\begin{array}{rl}
A[d] & \textrm{denotes the $d$-torsion subgroup of $A$},\\
A\{d\}& \textrm{denotes the $d$-primary part $\varinjlim\limits_{n} A [d^n]$ of $A$}.
 \end{array}\]

\section*{Acknowledgements}
This collaboration began at the workshop \emph{Rational Points 2017}. The authors are grateful to Michael Stoll for organising such a stimulating workshop.
FB thanks the Max-Planck-Institut f{$\ddot{\textrm{u}}$}r Mathematik for the financial support and for providing excellent working conditions.  The authors are grateful to Alexei Skorobogatov and Dami\'{a}n Gvirtz for useful conversations, to Daniel Loughran for pointing out the recent preprint \cite{Ieronymou}, and to Evis Ieronymou for his interest in this work. They are also indebted to the anonymous referee whose helpful comments improved this paper and its exposition. RN is supported by EPSRC grant EP/S004696/1.

\section{Some sets of varieties}\label{sec:Kummer}

Let $k$ be a field of characteristic 0. We define here our main objects of study.

\begin{defn}
Denote by $\overline{\scrR}_k$ the set of smooth, projective, geometrically integral, geometrically
rationally connected varieties over $k$. Denote by $\defi{K3}_k$ the set of K3 surfaces over $k$. (Note that being a K3 surface is a geometric property.) 
\end{defn}

\begin{defn}[{\cite[Defn 2.1]{SZ-KumVar}}] Let $A$ be an abelian variety over $k$ of dimension $\geq 2$. Any $k$-torsor $T$ under $A[2$]
gives rise to a $2$-covering $\rho: V \to A$, where $V$ is the quotient of $A\times_k T$ by the diagonal
action of $A[2]$ and $\rho$ comes from the projection onto the first factor. Then $T =\rho^{-1}(O_A)$ and $V$ has
the structure of a $k$-torsor under $A$. The class of $T$ maps to the class of $V$ under the
map $\mathrm{H}^1_{\et}(k, A[2]) \to \mathrm{H}^1_{\et}(k, A)$ induced by the inclusion of group schemes $A[2] \to A$ and, in
particular, the period of $V$ divides $2$.
Let $\sigma : \tilde{V} \to V$ be the blow-up of $V$ at $T \subset V$. The involution $[-1]: A \to A$ fixes $A[2]$
and induces involutions $\iota$ on $V$ and $\tilde{\iota}$ on $\tilde{V}$ whose fixed point sets are $T$ and the exceptional
divisor, respectively. We call $\Kum V  :=\tilde{V}/\tilde{\iota}$ the \defi{Kummer variety associated
to $V$ (or $T$)}. We remark that the quotient $\Kum V $ is geometrically isomorphic to 
$\Kum A$, so in particular it is smooth. A Kummer variety of dimension $2$ is called a Kummer surface.
\end{defn}

 \begin{defn}  We denote by $\overline{\scrK}_{k,g}$ the set of smooth, projective, geometrically integral varieties $X$ over $k$ of dimension $g$ such that $\overline{X}$ is a Kummer variety over $\overline{k}$, and by ${\scrK}_{k,g}$ the subset 
 of Kummer varieties over $k$ of dimension $g$. We let $\overline{\scrK}_k=\cup_{g\geq 2} \overline{\scrK}_{k,g}$ and, likewise, ${\scrK}_k=\cup_{g\geq 2} {\scrK}_{k,g}$. 
 \end{defn}

\begin{rmk} By \cite[Thm.~3]{Nikulin}, the set $\overline{\scrK}_{k,2}$ consists precisely of the varieties $X$ in $\defi{K3}_k$ such that $ \NS \Xbar$  contains a primitive sublattice isomorphic to the Kummer lattice. Moreover,
by \cite[Prop. 2.1]{VAV16}, there exists some absolute $M \in \Z_{\geq 1}$ (not depending on $k$) such that for each $X\in \overline{\scrK}_{k, 2}$, there is an extension $k_0/k$ of degree at most $M$ for which $X_{k_0}\in  \scrK_{k_0,2}$. In fact, the proof of  \cite[Prop. 2.1]{VAV16} shows that one can take $M= 2 \cdot \#\GL_{20}(\F_3)$. 
\end{rmk}

We will use properties of the underlying abelian varieties to deduce results for Kummer varieties. Therefore, it is convenient to define the following set.

 \begin{defn}
 Denote by $\overline{\scrA}_{k,g}$ the set of smooth, projective, geometrically integral varieties $X$ over $k$ of dimension $g$ such that $\overline{X}$ is an abelian variety over $\overline{k}$. Let $\overline{\scrA}_{k}=\cup_{g\geq 1}\overline{\scrA}_{k,g}$.
 \end{defn}

\section{Preliminaries on Brauer groups of products of varieties}\label{sec:prod}

In order to prove our main results, we need a finiteness result for the Brauer groups of products of varieties.

\begin{prop} \label{brfin} Let $k$ be a number field and let $X$ and $Y$ be smooth, projective, geometrically integral varieties over $k$.  If $\Br \Xbar^{\Gamma_k}$, $\Br_1 X/\Br_0 X$,  $\Br \Ybar^{\Gamma_k}$,  and $\Br_1 Y/\Br_0 Y$ are all finite, then so are  $\Br (\Xbar \times \Ybar)^{\Gamma_k}$, $\Br_1  (X \times Y)/\Br_0  (X \times Y)$, and, consequently,  $\Br (X \times Y)/\Br_0 (X \times Y)$.
\end{prop}
\begin{proof}  By \cite[Cor 3.2(i)]{SZ-prod}, $\Br (\Xbar \times \Ybar)^{\Gamma_k}$ is finite since both $\Br \Xbar^{\Gamma_k}$ and $\Br \Ybar^{\Gamma_k}$ are (by hypothesis). Using the natural injection
\[ \Br (X \times Y)/\Br_1 (X \times Y) \hookrightarrow \Br (\Xbar \times \Ybar)^{\Gamma_k},\]
it follows that $\Br (X \times Y)/\Br_1 (X \times Y)$ is finite. Moreover, by \cite[Remark 1.9]{SZ-prod}, the cokernel of the natural map
\[ \mathrm{H}_{\et}^1(k, \Pic \Xbar) \oplus \mathrm{H}_{\et}^1(k, \Pic \Ybar) \to \mathrm{H}_{\et}^1(k, \Pic(\Xbar \times \Ybar))\]
is finite. Since $k$ is a number field, the Hochschild-Serre spectral sequence 
\[\mathrm{H}^p_{\textrm{cts}}(\Gal(\kbar/k), \mathrm{H}_{\et}^q (\Xbar, \G_{m})) \Rightarrow \mathrm{H}_{\et}^{p+q}(X, \G_{m})\]
 yields $\mathrm{H}_{\et}^1(k, \Pic \Xbar) = \Br_1 X/ \Br_0 X$, and similarly  $\mathrm{H}_{\et}^1(k, \Pic \Ybar) = \Br_1 Y / \Br_0 Y$.  By hypothesis, $\Br_1 X / \Br_0 X$ and $\Br_1 Y / \Br_0 Y$  are finite, hence $\mathrm{H}_{\et}^1(k, \Pic(\Xbar \times \Ybar)) = \Br_1  (X \times Y)/\Br_0  (X \times Y)$ is also finite. Putting everything together, it follows that $\Br (X \times Y)/\Br_0 (X \times Y)$ is finite, as required.
\end{proof}

\begin{rmk} By induction on $n$, Proposition \ref{brfin} holds for $\prod_{i=1}^n X_i$, where the $X_i$ are smooth, projective, geometrically integral  varieties over $k$.
\end{rmk}

\begin{rmk} Let $l/k$ be an extension of number fields. By \cite[Cor.~2.8]{SZ-KumVar} and \cite[Thm.~1.2]{SZ-K3finite}, respectively, if $X\in \overline{\scrK}_k$ or $X\in\defi{K3}_k$, then both $\Br \Xbar^{\Gamma_l}$ and $\Br_1 X_l/\Br_0 X_l$ are finite. 
If $X/k$ is geometrically rationally connected, then $\Br\Xbar$ is finite and the finiteness of $\Br X_l/\Br_0 X_l$ is an easy consequence of the Hochschild-Serre spectral sequence -- see \cite[Prop 3.1.1]{Liang}, for example.
\end{rmk}

\begin{lemma} \label{brsplit}Let $k$ be a number field and let $X$ and $Y$ be smooth, projective, geometrically integral varieties over $k$.  Let $l/k$ be a finite extension. Then the identification $(X \times Y) (\A_l)=X(\A_l)\times Y(\A_l)$ induces natural inclusions 
 \[(X \times Y) (\A_l)^{\Br(X \times Y)} \subset  X(\A_l)^{\Br X} \times Y(\A_l)^{\Br Y}\]  and
 \[(X \times Y) (\A_l)^{\Br(X \times Y)\{d\}} \subset  X(\A_l)^{\Br X\{d\}} \times Y(\A_l)^{\Br Y\{d\}}\] 
 for all $d \in \Z_{>0}$.\end{lemma}
\begin{proof} Note that both the assignments $V \mapsto V(\A_l)^{\Br V \{d\}}$ and $V \mapsto V(\A_l)^{\Br V}$ from the category of varieties over $k$ to the category of sets are functorial. Hence, the result follows immediately by functoriality applied to the projections $X \times Y \to X$ and $X \times Y \to Y$.
\end{proof}
\begin{rmk} \label{rmk1} By induction on $n$, Lemma \ref{brsplit} holds for $\prod_{i=1}^n X_i$, where the $X_i$ are smooth, projective, geometrically integral  varieties over $k$.
\end{rmk}

\begin{rmk}
Note that in Lemma~\ref{brsplit} when we write $X(\A_l)^{\Br X}$ we are implicitly using the restriction map $\Res_{l/k}:\Br X\to \Br X_l$ to obtain a Brauer-Manin pairing $X(\A_l)\times \Br X\to \Q/\Z$. The image of $\Res_{l/k}$ may not be equal to $\Br X_l$, so Lemma~\ref{brsplit} does not follow from \cite[Thm.~C]{SZ-prod}. 
\end{rmk}

\section{Controlling the transcendental Brauer groups of geometrically abelian varieties and geometrically Kummer varieties over field extensions}\label{sec:ControlAbVar}

Let $k$ be a field of characteristic 0. In this section we show that the $d$-primary part of the transcendental Brauer group of a geometrically abelian or Kummer variety does not grow when the base field is extended, so long as the field extension satisfies certain conditions. 

\begin{prop}\label{prop:nogrowthabvar}
Let $d,n\in \Z_{>0}$. 
Let $X$ be a smooth, proper, geometrically integral variety over $k$. Let $l/k$ be a finite extension with degree coprime to $d$ such that $\Br\overline{X}[d^n]^{\Gamma_l}=\Br\overline{X}[d^n]^{\Gamma_k}$. 
Then
 the restriction map gives a canonical isomorphism of abelian groups
\[ \Res_{l/k}:(\Br X)[d^n]/(\Br_1 X)[d^n]\xrightarrow{\sim} (\Br X_l)[d^n]/(\Br_1 X_l)[d^n].\]
\end{prop}

\begin{proof}
\noindent\emph{Step 1:} We will show that the corestriction map induces an injection 
\[\Cor_{l/k} : (\Br X_l)[d^n]/(\Br_1 X_l)[d^n] \hookrightarrow (\Br X)[d^n]/(\Br_1 X)[d^n].\]
Let $B\in (\Br X_l)[d^n]$ and suppose that $\Cor_{l/k}B\in \Br_1 X$. By \cite[Lemme 1.4]{CTSk-2013}, 
\begin{equation}\label{eq:CorB}
\Res_{\overline{k}/k}\circ\Cor_{l/k}B=\sum_{\sigma\in\Gamma_k/\Gamma_l}\sigma(\Res_{\overline{k}/l}B),
\end{equation}
where the sum is over a set of coset representatives of $\Gamma_l$ in $\Gamma_k$. Since $\Cor_{l/k}B\in \Br_1 X$, the left-hand side equals zero. Moreover, $\Res_{\overline{k}/l}B\in\Br\overline{X}[d^n]^{\Gamma_l}$. By our hypothesis, this implies that $\Res_{\overline{k}/l}B\in\Br\overline{X}[d^n]^{\Gamma_k}$. Therefore, \eqref{eq:CorB} becomes
\begin{equation}\label{eq:BinBr1}
0=[l:k]\Res_{\overline{k}/l}B.
\end{equation}
Now recall that $[l:k]$ is coprime to $d$. Therefore, \eqref{eq:BinBr1} shows that $\Res_{\overline{k}/l}B=0$, since $B\in(\Br X_l)[d^n]$. Hence, $B\in(\Br_1 X_l)[d^n]$, as required.\\

\noindent\emph{Step 2:} Since $\Cor_{l/k}\circ\Res_{l/k}=[l:k]$ and $[l:k]$ is coprime to $d$, the corestriction map \[\Cor_{l/k}:(\Br X_l)[d^n]/(\Br_1 X_l)[d^n]\to (\Br X)[d^n]/(\Br_1 X)[d^n]\]
is surjective. Step 1 shows that it is an isomorphism and consequently the restriction map
\[\Res_{l/k}:(\Br X)[d^n]/(\Br_1 X)[d^n]  \to  (\Br X_l)[d^n]/(\Br_1 X_l)[d^n]\]
  is also an isomorphism.
\end{proof}

\begin{lemma}\label{lem:ppower}
Let $d \in \Z_{>0}$. 
Let $X\in\overline{\scrA}_{k}\cup \overline{\scrK}_{k}$. 
Let $l/k$ be a finite extension of degree coprime to $d$ such that the image of $\Gamma_l$ in $\Aut(\Br\overline{X}[d])$ is the same as that of $\Gamma_k$. 
Then for all $n\in\Z_{>0}$, the image of $\Gamma_l$ in $\Aut(\Br\overline{X}[d^n])$ is the same as that of $\Gamma_k$.
\end{lemma}

\begin{proof}
Factorise $d$ into primes, writing $d=p_1^{e_1}\dots p_r^{e_r}$. Then \[\Br\overline{X}[d^n]=\Br\overline{X}[p_1^{e_1n}]\times \dots \times \Br\overline{X}[p_r^{e_rn}]\] and \[\Aut(\Br\overline{X}[d^n])=\Aut(\Br\overline{X}[p_1^{e_1n}])\times \dots \times \Aut(\Br\overline{X}[p_r^{e_rn}]).\]
Thus, it suffices to prove the lemma when $d$ is a prime $p$.  Let $g := \dim X$.

\noindent\emph{Step 1:} First we show that there exists $m\in \Z_{\geq 0}$ such that for all $n\in \Z_{>0}$, $\Br\overline{X}[p^n]\cong (\Z/p^n\Z)^{m}$ as abelian groups. We begin by proving this in the case where $\overline{X}$ is an abelian variety, denoted by $\overline{A}$.

The Kummer exact sequence for $\overline{A}$ gives the following exact sequence of Galois modules:
\[0\to \NS\overline{A}/p^n\to \mathrm{H}^2_{\textrm{\'{e}t}}(\overline{A},\mu_{p^n})\to \Br\overline{A}[p^n]\to 0.\]

As an abelian group with no Galois action, \[\mathrm{H}^2_{\textrm{\'{e}t}}(\overline{A},\mu_{p^n})=\mathrm{H}^2_{\textrm{\'{e}t}}(\overline{A},\Z/p^n\Z)=\wedge^2\mathrm{H}^1_{\textrm{\'{e}t}}(\overline{A},\Z/p^n\Z).\]
It is well known that $\mathrm{H}^1_{\textrm{\'{e}t}}(\overline{A},\Z/p^n\Z)$ is isomorphic to $\Hom(\overline{A}[p^n],\Z/p^n\Z)$, see \cite[Lemma 2.1]{SZ-2012}, for example. As an abelian group, $\overline{A}[p^n]=(\Z/p^n\Z)^{2g}$ and hence \[\wedge^2\Hom(\overline{A}[p^n],\Z/p^n\Z)=\wedge^2(\Z/p^n\Z)^{2g}=(\Z/p^n\Z)^{g(2g-1)}.\]
Furthermore, $\NS\overline{A}$ is finitely generated and torsion-free of rank $\rho$ with $1\leq \rho\leq g^2$. Therefore, as an abelian group, $\Br\overline{A}[p^n]\cong (\Z/p^n\Z)^{m}$, where $m=g(2g-1)-\rho$.

Now if $X\in \overline{\scrK}_{k,g}$ then $\overline{X}=\Kum \overline{A}$ for some abelian variety $A$. Hence, by \cite[Proposition~2.7]{SZ-KumVar}, $\Br\overline{X}=\Br\overline{A}$. This completes Step 1.
  
\noindent\emph{Step 2:}  Since $\Br\overline{X}[p^n]\cong (\Z/p^n\Z)^{m}$ as abelian groups, we have $ \Aut(\Br\overline{X}[p^{n}])\cong\GL_m(\Z/p^n\Z)$.
The inclusion $\Br\overline{X}[p]\subset\Br\overline{X}[p^{n}]$ induces a surjective map
\[\alpha: \Aut(\Br\overline{X}[p^{n}])\to\Aut(\Br\overline{X}[p]),\] which can be identified with the map $\GL_m(\Z/p^n\Z)\to \GL_m(\Z/p\Z)$ given by reduction modulo $p$.  
Therefore, 
$\ker\alpha$ is a $p$-group.

Let $\phi:\Gamma_k\to \Aut(\Br\overline{X}[p^{n}])$ be the map giving the Galois action. The First Isomorphism Theorem gives
\[\#\phi(\Gamma_k)=\#(\phi(\Gamma_k)\cap\ker\alpha)\cdot \#\alpha(\phi(\Gamma_k))\] and \[
\#\phi(\Gamma_l)=\#(\phi(\Gamma_l)\cap\ker\alpha) \cdot \#\alpha(\phi(\Gamma_l)).\]
 Note that $\alpha\circ\phi :\Gamma_k\to \Aut(\Br\overline{X}[p])$ is the map giving the Galois action. Hence, $\alpha(\phi(\Gamma_k))=\alpha(\phi(\Gamma_l))$ by assumption. Therefore, 
\[[\phi(\Gamma_k):\phi(\Gamma_l)]=\frac{\#(\phi(\Gamma_k))}{\#(\phi(\Gamma_l))}=\frac{\#(\phi(\Gamma_k)\cap\ker\alpha)}{\#(\phi(\Gamma_l)\cap\ker\alpha)}.\] This is a power of $p$, since $\ker\alpha$ is a $p$-group. On the other hand, $[\phi(\Gamma_k):\phi(\Gamma_l)]$ divides $[\Gamma_k:\Gamma_l]=[l:k]$, which is coprime to $p$ by assumption.
Therefore, $[\phi(\Gamma_k):\phi(\Gamma_l)]=1$, as required.
\end{proof}

\begin{cor}\label{kum-ext3NU} 
Let $d\in \Z_{>0}$.
Let $X\in\overline{\scrA}_{k}\cup \overline{\scrK}_{k}$. Let $F/k$ be a finite extension such that $\Gamma_F$ acts trivially on $\Br\overline{X}[d]$. Let $l/k$ be a finite extension of degree coprime to $d$ which is linearly disjoint from $F/k$.
Then for all $n\in\Z_{>0}$,
\[ \Res_{l/k}:(\Br X)[d^n]/(\Br_1 X)[d^n]\to (\Br X_l)[d^n]/(\Br_1 X_l)[d^n]\]
is an isomorphism and consequently
\[ \Res_{l/k}:(\Br X/\Br_1 X)\{d\}\to (\Br X_l/\Br_1 X_l)\{d\}\]
is an isomorphism.
\end{cor}

\begin{proof}
Let $\phi:\Gamma_k\to \Aut(\Br\overline{X}[d])$ be the map giving the Galois action. Since $l/k$ is linearly disjoint from $F/k$, we have $\Gamma_k=\Gamma_l\Gamma_F$ and hence \[\phi(\Gamma_k)=\phi(\Gamma_l\Gamma_F)=\phi(\Gamma_l),\] since $\Gamma_F$ acts trivially on $\Br\overline{X}[d]$ by assumption.

 Now by Lemma~\ref{lem:ppower}, the image of $\Gamma_l$ in $\Aut(\Br\overline{X}[d^n])$ is the same as that of $\Gamma_k$, for all $n\in\Z_{>0}$. Now the first statement follows from Proposition~\ref{prop:nogrowthabvar}. Since $X$ is smooth, $\Br X$ is a torsion abelian group by \cite{Grothendieck}. Hence, the natural inclusion
\(\Br X\{d\}/\Br_1 X\{d\}\hookrightarrow (\Br X/\Br_1 X)\{d\}\)
is an equality, and similarly for $X_l$. Since
\[\frac{\Br X\{d\}}{\Br_1 X\{d\}}=  \frac{\varinjlim\limits_n\Br X[d^n]}{\varinjlim\limits_n\Br_1 X[d^n]}= \varinjlim\limits_n \frac{\Br X[d^n]}{\Br_1 X[d^n]},\]
the second statement follows from the first one.
\end{proof}

Next, we give a uniform version of Corollary \ref{kum-ext3NU} which applies to the whole of $\overline{\scrA}_{k,g} \cup \overline{\scrK}_{k,g}$ and every $n\in\Z_{>0}$. We will need a couple of auxiliary lemmas.

\begin{defn}\label{def:Q}
Let $d,m\in\Z_{>0}$ and let $d$ have prime factorisation $d=\prod_{i=1}^r p_i^{\epsilon_i}$. We define
$Q_{d,m}:=\{q\in\Z \textrm{ prime} : q\mid \prod_{i=1}^r p_i(p_i^m-1)\}$ and $N_{d,m}:=\prod_{q\in Q_{d,m}}q$.
\end{defn}

\begin{lemma}\label{lem:same action}
Let $M$ be a finite group such that $\Aut M$ has exponent dividing $e$. Let $l/k$ be a finite extension of degree coprime to $e$. Let $\phi:\Gamma_k\to\Aut M$ be a homomorphism making $M$ into a $\Gamma_k$-module. Then $\phi(\Gamma_k)=\phi(\Gamma_l)$.
\end{lemma}

\begin{proof}
We have $\phi(\Gamma_l) \leq \phi(\Gamma_k) \leq \Aut M$ with $[\phi(\Gamma_k):\phi(\Gamma_l)]$ dividing $[\Gamma_k:\Gamma_l]=[l:k]$ and with $\gcd([l:k], e) =\gcd([l:k], \exp(\Aut M)) =\gcd([l:k], \#\Aut M)=1$.
Since $M$ is finite, Lagrange's theorem gives
\[\#\Aut M = [\Aut M : \phi(\Gamma_l)] \cdot \# \phi(\Gamma_l)= [\Aut M : \phi(\Gamma_k)] \cdot \#\phi(\Gamma_k),\]
implying that  
\([\Aut M : \phi(\Gamma_l)] =   [\Aut M : \phi(\Gamma_k)] [\phi(\Gamma_k): \phi(\Gamma_l)].\)
Hence, $[\phi(\Gamma_k): \phi(\Gamma_l)]$ divides both  $\# \Aut M $ and $[l:k]$.
Since $\gcd([l:k], \#\Aut M)=1$, it follows that $[\phi(\Gamma_k): \phi(\Gamma_l)] = 1$ and thus that $\phi(\Gamma_k) = \phi(\Gamma_l)$, as required.
\end{proof}

\begin{lemma}\label{lem:lgood}Let $d\in\Z_{>0}$ and let $g\in\Z_{>1}$. 
Let $l/k$ be a finite extension of degree coprime to $N_{d,g(2g-1)-1}$. Then for all $X\in\overline{\scrA}_{k,g}\cup \overline{\scrK}_{k,g}$, and all $n\in\N$, the image of $\Gamma_l$ in $\Aut(\Br\overline{X}[d^n])$ is the same as the image of $\Gamma_k$.
\end{lemma}

\begin{proof}
The proof of Lemma~\ref{lem:ppower} shows that, as an abelian group with no Galois action, $\Br\overline{X}[d^n]\cong (\Z/d^n\Z)^{g(2g-1)-\rho}$, where $1\leq \rho\leq g^2$. 
By \cite[Thm. 4.1]{HillarRhea}, the primes dividing the exponent of $\Aut ((\Z/d^n\Z)^m)$ belong to $Q_{d,m}$.
 Now apply Lemma~\ref{lem:same action} with $e$ a suitable power of $N_{d,g(2g-1)-1}$ to see that the image of $\Gamma_k$ in $\Aut(\Br\overline{X}[d^n])$ is the same as the image of $\Gamma_l$.
\end{proof}

\begin{cor}\label{cor:kum-ext3}Let $d\in\Z_{>0}$ and let $g\in\Z_{>1}$. 
Then for all  $X\in\overline{\scrA}_{k,g}\cup \overline{\scrK}_{k,g}$, all $n \in \Z_{>0}$, and all finite extensions $l/k$ of degree coprime to $N_{d,g(2g-1)-1}$,
\[ \Res_{l/k}:(\Br X)[d^n]/(\Br_1 X)[d^n]\to (\Br X_l)[d^n]/(\Br_1 X_l)[d^n]\]
is an isomorphism and
consequently \[  \Res_{l/k}:(\Br X/\Br_1 X)\{d\} \to (\Br X_l/\Br_1 X_l)\{d\}\]
is an isomorphism.
\end{cor}

\begin{proof}
Clearly, if $[l:k]$ is coprime to $N_{d,g(2g-1)-1}$, then it is coprime to $d$.
Now Lemma~\ref{lem:lgood} allows us to apply Proposition~\ref{prop:nogrowthabvar}.
\end{proof}

\section{Controlling the  Brauer groups of geometrically Kummer varieties over field extensions}\label{sec:ControlKummer}

Throughout this section, $k$ will be a number field.  Here we give sufficient criteria for the $d$-primary part of the algebraic part of the Brauer group of a Kummer variety $X$ to be unchanged under an extension of the base field. We then combine this with the results of Section~\ref{sec:ControlAbVar} to give sufficient criteria for the $d$-primary part of the quotient of the Brauer group of a Kummer variety $X$ by the subgroup $\Br_0X$ of constant Brauer elements to be unchanged under an extension of the base field.

 The following well-known lemma tells us how to control the algebraic part of the Brauer group over finite extensions of the base field.
 
 \begin{lemma}\label{lemBr0} Let $X$ be a smooth, proper, geometrically integral variety over $k$ with $\Pic \Xbar$ free of rank $r$. Let $k'/k$ be a finite Galois extension such that $\Pic X_{k'}=\Pic \Xbar $.  If a finite extension $l/k$ is linearly disjoint 
 from $k'/k$, then 
the natural homomorphism \[\Res_{l/k}: \Br_1 X/\Br_0 X \to \Br_1 X_{l}/\Br_0 X_{l}\]
 is an isomorphism. 
 \end{lemma}

 \begin{proof} We have $ \mathrm{H}^1_{\et}(k, \Pic \Xbar)  =  \mathrm{H}^1_{\textrm{cts}}(\Gal(k'/k), \Pic X_{k'})$ and, writing $l':=k' l$ for the compositum field,  $  \mathrm{H}^1_{\textrm{cts}}(\Gal(l'/l), \Pic X_{l'}) =  \mathrm{H}^1_{\et}(l, \Pic \Xbar)$.  Moreover, the restriction map induces an isomorphism
\[
\Res_{l/k}:  \mathrm{H}^1_{\textrm{cts}}(\Gal(k'/k), \Pic X_{k'}) \xrightarrow{\sim}\mathrm{H}^1_{\textrm{cts}}(\Gal(l'/l),\Pic X_{l'}).
 \]
The Hochschild-Serre spectral sequence 
\[\mathrm{H}^p_{\textrm{cts}}(\Gal(\kbar/k), \mathrm{H}_{\et}^q (\Xbar, \G_m)) \Rightarrow \mathrm{H}_{\et}^{p+q}(X, \G_m)\]
 yields $\Br_1 X/\Br_0 X = \mathrm{H}^1_{\et}(k, \Pic \Xbar)$, and similarly $\Br_1 X_l/\Br_0 X_l = \mathrm{H}^1_{\et}(l, \Pic \Xbar)$, whence the result. 
 \end{proof}

 We also have a  uniform version of  Lemma \ref{lemBr0}:

  \begin{lemma}\label{lemBr0U} Let $\frakF_{k,r}$ be the set of all smooth, proper, geometrically integral varieties over  $k$ with $\Pic \Xbar$ free of rank $r$.  If a finite extension $l/k$ has degree coprime to $\# \GL_r(\F_3)$, then 
the natural homomorphism \[\Res_{l/k}: \Br_1 X/\Br_0 X \to \Br_1 X_{l}/\Br_0 X_{l}\]
 is an isomorphism for all $X \in \frakF_{k,r}$. 
 \end{lemma}
 
 \begin{proof} Let $X \in \frakF_{k,r}$. Since $\Gamma_k$ acts on $\Pic \Xbar$, we have a map $\phi_X: \Gamma_k \to \Aut (\Pic \Xbar)$. Let $k'_X$ be the fixed field of $\ker \phi_X$. Then $[k'_X:k] = \# \im \phi_X$. Since $\Pic \Xbar$ is isomorphic to $\Z^r$ as an abelian group, $\im \phi_X$ is isomorphic to a finite subgroup of $\GL_r(\Z)$. By a result of Minkowski (see \cite[p.197]{Minkowski}), any such subgroup is isomorphic to a subgroup of $\GL_r(\F_3)$. Therefore, $[k'_X:k]$ divides $\# \GL_r(\F_3)$. Since $l/k$ is coprime to $\# \GL_r(\F_3)$, it follows that $[l:k]$ is coprime to $[k'_X:k]$ and hence that $l/k$ and $k'_X/k$ are linearly disjoint. Therefore, we can apply Lemma \ref{lemBr0} to deduce that \(\Res_{l/k}: \Br_1 X/\Br_0 X \to \Br_1 X_{l}/\Br_0 X_{l}\) is an isomorphism.
 \end{proof}
 
  \begin{rmk} \label{rem:kumfreerank}
 If $X\in\overline{\scrK}_{k,g}$, then $\Pic \Xbar$ is a finitely-generated free $\Z$-module of rank $\leq 2^{2g} + g^2$ (see \cite[Cor. 2.4]{SZ-KumVar}).
 \end{rmk}

 \begin{theorem}\label{thmisoNU}
  Let $d,n\in\Z_{>0}$, let $g\in\Z_{>1}$ and let $X\in   \overline{\scrK}_{k,g}$.
 Fix a finite Galois extension $k'/k$ such that $\Pic X_{k'}=\Pic \Xbar $. Let $F/k$ be a finite extension such that $\Gamma_F$ acts trivially on $\Br\overline{X}[d]$.
 Let $l/k$ be a finite extension of degree coprime to $d$ which is linearly disjoint from $Fk'/k$. Then 
\[ \Res_{l/k}:(\Br X)[d^n]/(\Br_0 X)[d^n]\rightarrow (\Br X_l)[d^n]/(\Br_0 X_l)[d^n]\]
 is an isomorphism and consequently
 \[\Res_{l/k}:(\Br X/\Br_0 X)\{d\}\to(\Br X_l/\Br_0 X_l)\{d\}\]
 is an isomorphism.
\end{theorem}

\begin{proof} We have the following commutative diagram with exact rows
 \[ \begin{tikzpicture}
  \matrix (m) [matrix of math nodes, row sep=1.5em,
    column sep=3em]{
0 &\frac{(\Br_1 X)[d^n]}{(\Br_0 X)[d^n]}      & \frac{(\Br X)[d^n]}{(\Br_0 X)[d^n]}      &  \frac{(\Br X)[d^n]}{(\Br_1 X)[d^n]}     & 0\\
0 &\frac{(\Br_1 X_{l})[d^n]}{(\Br_0 X_{l})[d^n]}     & \frac{(\Br X_{l})[d^n]}{(\Br_0 X_{l})[d^n]}       &  \frac{(\Br X_{l})[d^n]}{(\Br_1 X_{l})[d^n]}      & 0,\\
};
  \path[-stealth]
  (m-1-1) edge (m-1-2) 
   (m-1-2) edge (m-1-3)
   (m-1-3) edge (m-1-4)
   (m-1-4) edge (m-1-5) 
     (m-2-1) edge (m-2-2) 
   (m-2-2) edge (m-2-3)
   (m-2-3) edge (m-2-4)
   (m-2-4) edge (m-2-5) 
  (m-1-2) edge node [left] {\tiny$\Res_{l/k}$}  node [above, sloped] {$\sim$}(m-2-2) 
   (m-1-3) edge node [left] {\tiny$\Res_{l/k}$}(m-2-3) 
   (m-1-4) edge node [left] {\tiny$\Res_{l/k}$} node [above, sloped] {$\sim$}(m-2-4) 
;
\end{tikzpicture}\]
where the  first and third vertical arrows are isomorphisms  by Lemma \ref{lemBr0} and by  Corollary~\ref{kum-ext3NU}, respectively. It follows from the commutativity of the diagram that the middle vertical arrow is also an isomorphism.
\end{proof}

We also have the following uniform version of Theorem~\ref{thmisoNU}:

\begin{theorem} \label{thmiso} Fix $d \in \Z_{>0}$ and $g \in \Z_{>1}$. 
Then, for all  $X \in   \overline{\scrK}_{k,g}$, all $n\in\Z_{>0}$, and all finite extensions $l/k$ of degree coprime to $\# \GL_{2^{2g} + g^2}(\F_3)\cdot N_{d,g(2g-1)-1}$,
the natural homomorphism  
  \[ \Res_{l/k}:(\Br X)[d^n]/(\Br_0 X)[d^n]\rightarrow (\Br X_l)[d^n]/(\Br_0 X_l)[d^n]\]
 is an isomorphism and consequently 
\[\Res_{l/k}: (\Br X/\Br_0 X)\{d\} \to (\Br X_{l}/\Br_0 X_{l})\{d\}\]
 is an isomorphism. 
\end{theorem}

\begin{proof} This follows from Lemma \ref{lemBr0U} together with Remark  \ref{rem:kumfreerank} and Corollary~\ref{cor:kum-ext3}.
\end{proof}


\section{Preliminaries on zero-cycles} 
\label{sec:0-cyc}

In this section $k$ will always be a number field. Recall the construction of the Brauer-Manin set for 0-cycles.
Let $X$ be a smooth, geometrically integral variety over $k$. For any field extension $l/k$, we denote by $\underline{X}_l$ the set of closed points on $X_l$. Let $Z_0(X_l)$ denote the set of 0-cycles of $X$ over a field extension $l/k$, that is, the set of formal sums 
\(\sum_{x \in \underline{X}_l} n_x x\)
 where the integers $n_x$ are zero for all but finitely many $x \in  \underline{X}_l$.   If $\delta \in \Z_{\geq 0}$, we denote by $Z_0^\delta(X_l) \subset Z_0(X_l)$ the subset of 0-cycles of degree $\delta$, that is, 0-cycles $z:=  \sum_{x \in \underline{X}_l} n_x x$ such that $\deg(z):= \sum_{x \in \underline{X}_l} [l(x):l] n_x = \delta$. We can extend the Brauer-Manin pairing to 0-cycles of degree $\delta$ by defining, for any subset $B \subset \Br X$, 
\[ 
\begin{array}{lll}
\langle \ \ , \ \  \rangle_{\textrm{BM}}: &\prod_{v \in \Omega_k } Z^\delta_0(X_{k_v}) \times B & \to \Q/\Z\\
& \left( \left( \sum_{x_v \in \underline{X}_{k_v}} n_{x_v} x_v \right)_{v },  \alpha \right) & \mapsto  \sum_{v } \inv_v\left(  \sum_{x_v \in \underline{X}_{k_v}} n_{x_v} \Cor_{k_v(x_{v})/k_v}\alpha(x_v) \right).
\end{array} \]

\begin{defn} Let $X$ be a smooth, geometrically integral variety over  $k$. Let $B \subset \Br X$ be a subset of the Brauer group. We define the \defi{Brauer-Manin set associated to $B$ for 0-cycles of degree $\delta$}, denoted by $Z^{\delta}_0(X_{\A_k})^{B}$, to be the left kernel of the Brauer-Manin pairing above.
\end{defn}

\begin{defn} \label{defn:BrHP} Let $\{X_\omega\}_\omega$ be a family of smooth, geometrically integral varieties over $k$. For each $X_\omega$, let $B_\omega \subset \Br X_\omega$ be a subset of the Brauer group. We say that \defi{the $\{B_\omega\}_\omega$-obstruction to the Hasse principle for 0-cycles of degree $\delta$ is the only one for the family $\{X_\omega\}_\omega$} if $Z^{\delta}_0(X_{\omega, \A_k})^{B_\omega} \neq \emptyset$ implies $Z_0^\delta(X_\omega) \neq \emptyset$, for all $X_\omega$. 
\end{defn}

The following definition is taken from \cite{Liang} and slightly differs from the definition given in e.g. \cite{CTSD94}.
\begin{defn}[{\cite{Liang}}] \label{defn:BrWA} Let $\{X_\omega\}_\omega$ be a family of smooth, geometrically integral varieties over $k$. For each $X_\omega$, let $B_\omega \subset \Br X_\omega$ be a subset of the Brauer group.  We say that \defi{the $\{B_\omega\}_\omega$-obstruction to weak approximation for 0-cycles of degree $\delta$ is the only one}  if for any $n \in \Z_{>0}$, for any finite subset $S \subset \Omega_k$, and for any $(z_v)_{v \in \Omega_k} \in Z_0^\delta(X_{\omega, \A_k})^{B_\omega}$, there exists some $z_{n,S} \in Z_0^\delta(X_\omega)$ such that $z_{n,S}$ and $z_v$ have the same image in $\CH_0(X_{\omega, k_v})/n$ for all $v \in S$, for all $X_\omega$, where $\CH_0$ denotes the usual Chow group of 0-cycles.  
\end{defn}

\begin{rmk} In Definitions \ref{defn:BrHP} and  \ref{defn:BrWA}, when for all $X_\omega$ we take $B_\omega := \Br X_\omega$ or $B_\omega := \Br X_\omega \{d\}$ for some $d \in \Z_{>0}$, we say, respectively, that the Brauer-Manin obstruction or the $d$-primary Brauer-Manin obstruction is the only one for $0$-cycles of degree $\delta$ for the family $\{ X_\omega\}_\omega$.
\end{rmk}

The main conjecture governing the arithmetic properties of 0-cycles of degree 1 is the following:

\begin{conj}[{Colliot-Th\'el\`ene}] Let $\{X_\omega\}_\omega$ be the family of \emph{all} smooth, proper, geometrically integral varieties over $k$. Then the Brauer-Manin obstruction to weak approximation for 0-cycles of degree $1$ is the only one for $\{X_\omega\}_\omega$.
\end{conj}

\section{From rational points to zero-cycles} \label{sec:pfs}

In this section, we prove our main results, which are various analogues of \cite[Thm. 3.2.1]{Liang} wherein we relate the Brauer-Manin obstruction for rational points to that for 0-cycles on products of (geometrically) Kummer varieties, K3 surfaces and rationally connected varieties. 
We begin with a lemma that allows us to extend the base field without affecting the Brauer-Manin obstruction. Throughout this section, $k$ will always be a number field.

\begin{lemma}\label{lem:nogrowth}
\begin{enumerate}
\item\label{rat}If $X\in\overline{\scrR}_k$, then there exists a finite extension $F_X/k$ such that for all finite extensions $\ell/k$ that are linearly disjoint from $F_X/k$, 
\[\Res_{\ell/k}:\Br X/\Br_0 X\to\Br X_\ell/\Br_0 X_\ell\]
is an isomorphism.

\item\label{K3}If $X\in\defi{K3}_k$, then for all $n\in\Z_{>0}$, there exists a finite extension $F_{X,n}$ such that for all finite extensions $\ell/k$ with $[\ell:k]\leq n$ that are linearly disjoint from $F_{X,n}/k$, 
\[\Res_{\ell/k}:\Br X/\Br_0 X\to\Br X_\ell/\Br_0 X_\ell\]
is an isomorphism.

\item\label{Kum}If $X\in \overline{\scrK}_k$, then there exists a finite extension $F_X/k$ such that for any $d\in\Z_{>0}$ and all finite extensions $\ell/k$ of degree coprime to $d$ that are linearly disjoint from $F_X/k$, 
\[\Res_{\ell/k}:(\Br X/\Br_0 X)\{d\}\to(\Br X_\ell/\Br_0 X_\ell)\{d\}\]
is an isomorphism. 
\end{enumerate}

\end{lemma}

\begin{proof}
For $X\in \overline{\scrR}_k$, this is \cite[Prop. 3.1.1]{Liang}. 
For $X\in\defi{K3}_k$, see the proof of \cite[Thm. 1.1]{Ieronymou}.
For $X\in\overline{\scrK}_k$, noting that $\Br\overline{X}[d]$ is finite, let $F/k$ be a finite extension such that $\Gamma_F$ acts trivially on $\Br\overline{X}[d]$. Let $k'/k$ be a finite Galois extension such that $\Pic X_{k'}=\Pic\overline{X}$. Now Theorem~\ref{thmisoNU} shows that we can take $F_X=Fk'$.
\end{proof}

We are now ready to prove our main result, which is Theorem~\ref{thm:mixed}. 

\begin{rmk}\label{rmk:suffclose}
In the proof of Theorem~\ref{thm:mixed}, we use a notion of two 0-cycles on a variety $X$ over $k_v$ being ``sufficiently close'' as described by Liang in \cite[\S 1.1]{Liang}. In particular, by \cite[Lemme 1.8]{Wittenberg12}, if two 0-cycles on $X$ are sufficiently close then they have the same image in $\CH_0(X)/n$. Furthermore, since the evaluation maps are locally constant, if $b\in\Br(X)$ and two 0-cycles $z$ and $z'$ on $X$ are sufficiently close, then $b$ has the same evaluation at $z$ and $z'$.
\end{rmk}

 \begin{proof}[Proof of Theorem~\ref{thm:mixed}] We follow closely the strategy of the proof of \cite[Thm. 3.2.1]{Liang} and give here just a sketch of the proof; for more details, we refer the reader to Liang's paper \cite{Liang}. We prove the result for weak approximation; the argument for the Hasse principle is similar.
\begin{enumerate}
\item    Let $(\zeta_v)_v \in Z_0^\delta(W_{\A_k}) $ be orthogonal to $\Br(W)\{c\} $.  Consider the projection $\pr : W \times \PP^1_k \to W$ with section $\sigma : w \mapsto (w,u_0)$,  where $u_0 \in \PP^1_k(k)$ is some fixed rational point. We note that $\pr^\ast: \Br W \to \Br(W \times \PP^1_k)$ is an isomorphism. Let $\sigma_\ast: Z^{\delta}_0(W) \to Z^{\delta}_0(W \times \PP^1_k)$ be the induced map on zero-cycles. Then $(\sigma_\ast(\zeta_v))_v$ on $W \times \PP^1_k$ is orthogonal to $\Br(W \times \PP^1_k)\{c\}$. To ease notation, we let $z_v:= \sigma_\ast(\zeta_v)$ for all $v \in \Omega_k$.  Let $\{ b_i\}_{i=1}^s$  be a complete set of representatives for $  (\Br (W\times \PP^1_k) /\Br_0(W\times \PP^1_k)) \{c\} $, which is finite by Proposition \ref{brfin}. Let $m:=\#(\Br W /\Br_0W)\{c\}$.  We fix $n \in \Z_{>0}$ and a finite subset $S \subset \Omega_k$.  We also fix a closed point $P:=(w_0, u_0) \in W \times \PP_k^1$ and let $\delta_P :=[k(P):k]$.

    \item Let $S_0 \subset \Omega_k$ be a finite subset containing  $S$ and
    all the archimedean places of $k$. By enlarging $S_0$ if necessary, we can assume that, for any $v \not\in S_0$, 
\begin{itemize}
\item $\langle z'_v, b_i \rangle_{\textrm{BM}} = 0$ for any $i \in \{ 1, ..., s\}$ and for any $z'_v \in Z_0(W_{k_v} \times \PP^1_{k_v})$, and
\item $W_{k_v}$ has a $k_v$-point $w_v$ (this follows from the Lang-Weil estimates and Hensel's lifting).
\end{itemize}

\item For each $v \in S_0$, write $z_v = z_v^+- z_v^-$, where $z_v^+$ and $z_v^-$ are effective 0-cycles with disjoint supports, and let 
\(z_v^1:= z_v +  dm n\delta_P z_v^-.\)
Note that
\(\deg(z_v^1) = \delta +  dm n   \delta_P \deg(z_v^-).\) By adding to each $z_v^1$ a suitable multiple of the 0-cycle $ dm n   P_v$, where $P_v:=P \otimes_k k_v$, we obtain 0-cycles $z_v^2$ of the same degree $\Delta > 0$ for all $v \in S_0$, where we can take $\Delta$ to satisfy $\Delta \equiv \delta  \mod  dm n   \delta_P$. Using the natural projection $\pi: W_{k_v} \times \PP_{k_v}^1 \to \PP_{k_v}^1$ and a moving lemma by Liang (\cite[Lemma 1.3.1]{Liang}), for each $v \in S_0$ we can find an effective 0-cycle $z_v^3$ of degree $\Delta$ such that $\pi_\ast(z_v^3)$ is separable and $z_v^3$ is sufficiently close to $z_v^2$. By Remark~\ref{rmk:suffclose}, $z_v^3$ has the same image in $\CH_0(W_{k_v})/n$ as $z_v^2, z_v^1$ and $z_v$. 

        \item For each $i = 1, \dots , r$, let $F_{X_i}$ be as in Lemma~\ref{lem:nogrowth}, with $F_{X_i}:=F_{X_i,\Delta}$ if $X_i\in\defi{K3}_k$, and let $F:=F_{X_1}\dots F_{X_r}$ be the compositum.
    
\item By \cite[Prop. 3.3.3]{Liang} and the discussion in the proof of  \cite[Prop. 3.4.1]{Liang}, we can find a closed point $\theta \in \PP^1_k$ sufficiently close, for $v \in S_0$, to $\pi_\ast(z_v^3)$ such that $[k(\theta):k] = \Delta$  
and $k(\theta)/k$ is linearly disjoint from $F/ k$. Following the proof of \cite[Thm. 3.2.1]{Liang}, making the identification $\pi^{-1}(\theta) \cong W_{k(\theta)}$  we also obtain an adelic point $(\mathcal{M}_w)_w \in W_{k(\theta)}(\A_{k(\theta)})$ such that $\sum_{w\mid v}\mathcal{M}_w$ is sufficiently close to $z_v^3$ (and hence to $z^2_v$) for all $v\in S_0$. Hence, $\sum_{w\mid v}\mathcal{M}_w$ has the same image as $z_v$ in $\CH_0(W_{k_v})/n$ for all $v \in S$,  and also
$(\mathcal{M}_w)_w \in W_{k(\theta)}(\A_{k(\theta)})^{\Br W \{c\}}$.

\item Via Lemma \ref{brsplit} and Remark~\ref{rmk1}, we view $(\mathcal{M}_w)_w$ as an adelic point $ ((\mathcal{M}_{1,w})_w,  \dots,  (\mathcal{M}_{r,w})_w) $ in  $X_{1, k(\theta)}(\A_{k(\theta)})^{\Br X_1 \{c\}}\times \dots\times X_{r, k(\theta)}(\A_{k(\theta)})^{\Br X_r\{c\}}$.
Hence, we have an adelic point $(\mathcal{M}_{i,w})_w$ in $X_{i, k(\theta)}(\A_{k(\theta)})^{\Br X_i \{c\} }$ for each $i = 1, \dots , r$.

\item 
Recall that $k(\theta)/k$ is linearly disjoint from $F_{X_i}$ for all $i=1,\dots ,r$ and also $[k(\theta):k]=\Delta$. For $X_i\in \overline{\scrR}_k\cup \defi{K3}_k$, Lemma~\ref{lem:nogrowth} shows that $\Res_{k(\theta)/k}:\Br X_i/\Br_0X_i\to \Br X_{i,k(\theta)}/\Br_0X_{i,k(\theta)}$ is an isomorphism. Furthermore, since $a_i\mid c$, we have $\Br X_i/\Br_0X_i=(\Br X_i/\Br_0X_i)\{c\}$. Therefore, $(\mathcal{M}_{i,w})_w\in X_{i, k(\theta)}(\A_{k(\theta)})^{\Br X_i \{c\} }=X_{i, k(\theta)}(\A_{k(\theta)})^{\Br X_{i,k(\theta)}}$.

By assumption, the Brauer-Manin obstruction is the only obstruction to weak approximation for rational points on $X_{i,k(\theta)}$. Hence, there exists a global point $\mathcal{M}_i\in X_{i, k(\theta)}(k(\theta))$ sufficiently close to $\mathcal{M}_{i,w}$ for all places $w$ of $k(\theta)$ lying above places in $S$.

\item Now suppose that for some $i$ we have $X_i\in \overline{\scrK}_k$. Since $\overline{\scrK}_{k,2}\subset \defi{K3}_k$, if $X_i\in \overline{\scrK}_{k,2}$ then we are in the previous step. If $X_i\in\overline{\scrK}_{k,g}$ for some $g\geq 3$  then, by assumption, $\delta$ is coprime to $d$. Since $\Delta \equiv \delta  \mod  dm n  \delta_P$, this implies that $\Delta$ is coprime to $d$ and Lemma~\ref{lem:nogrowth} gives 
      $X_{i, k(\theta)}(\A_{k(\theta)})^{\Br X_i\{d\} }=X_{i, k(\theta)}(\A_{k(\theta)})^{\Br X_{i,k(\theta)}\{d\} }$. Since $(\calM_{i,w})_w \in X_{i, k(\theta)}(\A_{k(\theta)})^{\Br X_i\{c\} }\subset X_{i, k(\theta)}(\A_{k(\theta)})^{\Br X_i\{d\} }$, we obtain $(\calM_{i,w})_w \in   X_{i, k(\theta)}(\A_{k(\theta)})^{\Br X_{i,k(\theta)}\{d\} }$.
      
      By assumption, the $d$-primary Brauer-Manin obstruction is the only obstruction to weak approximation for rational points on $X_{i,k(\theta)}$. Hence, there exists a global point $\mathcal{M}_i\in X_{i, k(\theta)}(k(\theta))$ sufficiently close to $\mathcal{M}_{i,w}$ for all places $w$ of $k(\theta)$ lying above places in $S$.

\item It follows that the $k(\theta)$-point $\mathcal{M}:=(\mathcal{M}_1,\dots, \mathcal{M}_r)\in W_{k(\theta)}(k(\theta))$  is sufficiently close to $\mathcal{M}_w$ for all places $w$ of $k(\theta)$ lying above places in $S$,  implying  that $\mathcal{M}$ and $z_v$ have the same image in $\CH_0(W_{k_v})/n$ for all $v \in S$.

\item When viewed as a 0-cycle on $W$, $\mathcal{M}$ has degree $\Delta$. Since  $\Delta \equiv \delta \mod dm n \delta_P$, adding a suitable multiple of the degree $\delta_P$ closed point $w_0 = \pr(P)$ to $\mathcal{M}$  yields a global 0-cycle of degree $\delta$ on $W$ with the same image as $z_v$ in $\CH_0(W_{k_v})/n$ for all $v \in S$, as required. \qedhere
\end{enumerate}
\end{proof}

 \begin{rmk}
 Let $K/k$ be a finite extension and let $N\in \Z_{>0}$ be coprime to $\delta$. The proof above shows that Theorem~\ref{thm:mixed} still holds with the weaker assumptions obtained by replacing the phrase `for all finite extensions $l/k$' with `for all finite extensions $l/k$ of degree coprime to $N$ that are linearly disjoint from $K/k$'. The same goes for Theorem~\ref{thm:Kummer2-primary} and Theorem~\ref{thm:mixed2} below. 
 \end{rmk}

\begin{rmk} The product of two or more Kummer varieties is not, in general, a Kummer variety. Similarly, the product of two or more K3 surfaces is not a K3 surface. On the other hand, the product of two or more geometrically rationally connected varieties is again geometrically rationally connected.
\end{rmk}

 \begin{rmk} It is clear that the strategy of the proof of Theorem \ref{thm:mixed} could be used to deal with other (products of) varieties, provided that enough information on the Brauer groups and on how to control them over certain field extensions is known. 
 \end{rmk}

\begin{rmk} \label{rem2} Let $X$ be a Kummer variety over $k$.  By \cite[Thm. 1.7]{CreutzViray-DegBMO}, it follows that $X_l(\A_l)^{\Br X_l} \neq \emptyset$ if and only if $X_l(\A_l)^{\Br X_l\{2\}} \neq \emptyset$, for any finite extension $l/k$.  Hence, for any such $X$,  saying  ``the 2-primary Brauer-Manin obstruction to the Hasse principle is the only one for rational points on $X_l$'' is equivalent to saying ``the Brauer-Manin obstruction to the Hasse principle is the only one for rational points on $X_l$''. 
\end{rmk}

As a special case of Theorem \ref{thm:mixed}, we obtain the following result for the Hasse principle:

\begin{thm}\label{thm:mixed2}
  Let $W:= \prod_{i=1}^r X_i$, where $X_i \in \scrK_{k} \cup \defi{K3}_k \cup \overline{\scrR}_k$ for all $i \in\{ 1, \dots, r\}$. Let $\delta \in \Z$ and assume that $\delta$ is odd if at least one of the $X_i$ is in $\scrK_{k,g}$ with $g \geq 3$.  Suppose that, for all $i \in\{ 1, \dots, r\}$ and all finite extensions $l/k$, the Brauer-Manin obstruction is the only obstruction to the Hasse principle for rational points on $X_{i,l}$. Then 
the Brauer-Manin obstruction is the only obstruction to the Hasse principle for 0-cycles of degree $\delta$ on $W$. 
\end{thm}
\begin{proof}
Take $d=2$ in Theorem~\ref{thm:mixed}. By \cite[Thm. 1.7]{CreutzViray-DegBMO}, the assumption that the Brauer-Manin obstruction to the Hasse principle is the only one for rational points on $X_l\in\scrK_l$  implies that the $2$-primary Brauer-Manin obstruction to the Hasse principle is the only one for rational points on $X_l$.
\end{proof}

\begin{rmk}
In Theorem~\ref{thm:mixed2}, we cannot include geometrically Kummer varieties of dimension greater than $2$ unless they are Kummer varieties over $k$ because \cite[Thm. 1.7]{CreutzViray-DegBMO} only applies to Kummer varieties, not to the larger class of geometrically Kummer varieties. However, geometrically Kummer surfaces are included because they are K3 surfaces.
\end{rmk}

\begin{proof}[Proof of Theorem~\ref{thm:Kummer2-primary}]
This is an immediate consequence of \cite[Thm. 1.7]{CreutzViray-DegBMO} and the proof of Theorem~\ref{thm:mixed} applied with $d=2$. 
\end{proof}

In \cite{Skorobogatov-K3Conj}, Skorobogatov made the following conjecture:

\begin{conj} [Skorobogatov] \label{conjsko} The Brauer-Manin obstruction is the only obstruction to the Hasse
principle and weak approximation on K3 surfaces over number fields.
\end{conj}

\begin{proof}[Proof of Corollary~\ref{cor:conditional}]
Run the proof of Theorem~\ref{thm:mixed} for $d=2$, taking into account Remark~\ref{rem2}.
\end{proof}

 It is natural to ask the following question:
 \begin{question}\label{q} Is the Brauer-Manin obstruction the only obstruction to the Hasse principle for the class $\scrK_{k}$ of Kummer varieties over a number field $k$? Or indeed for the class $\overline{\scrK}_{k}$ of geometrically Kummer varieties over a number field $k$? Is the Brauer-Manin obstruction the only obstruction to weak approximation for either of these classes?
 \end{question}

By \cite[Thm. 1.7]{CreutzViray-DegBMO}, the first part of Question~\ref{q} is equivalent to asking whether the $2$-primary Brauer-Manin obstruction is the only obstruction to the Hasse principle for the class $\scrK_{k}$ of Kummer varieties over a number field $k$.

As a corollary of Theorem \ref{thm:mixed2}, we obtain a result that is conditional on Conjecture~\ref{conjsko}, a positive answer to the first part of our Question~\ref{q}, and the following conjecture of Colliot-Th\'{e}l\`{e}ne (stated in full generality in \cite[p. 174]{CT03}):

\begin{conj}[Colliot-Th\'{e}l\`{e}ne]\label{conjrat}
The Brauer-Manin obstruction is the only obstruction to the Hasse principle for rational points for the class $\overline{\scrR}_k$ of all smooth, projective, geometrically integral, geometrically rationally connected varieties over a number field $k$.
\end{conj}

 \begin{cor} \label{cor:mainK3R}
 Let $k$ be a number field. Let $X_i \in {\scrK}_k \cup \defi{K3}_k\cup \overline{\scrR}_k$ for all $i = 1, ..., r$  and write $W:= \prod_{i=1}^r X_i$. Let $\delta \in \Z$ and assume that $\delta$ is odd if at least one of the $X_i$ is in $\scrK_{k,g}$ with $g \geq 3$.  If Conjectures \ref{conjsko} and \ref{conjrat} are true and if the first part of our Question \ref{q} has a positive answer, then the Brauer-Manin obstruction is the only obstruction to the Hasse principle for 0-cycles of degree $\delta$ on $W$. 
 \end{cor}

\section{Transferring emptiness of the Brauer-Manin set over field extensions
}\label{sec:transfer}
Let $k$ be a number field.
In the previous section, we proved that the Brauer-Manin obstruction to the Hasse principle is the only one for 0-cycles on certain varieties $X/k$, assuming that the Brauer-Manin obstruction to the Hasse principle on $X_l$ is the only one for rational points for sufficiently many finite extensions $l/k$. One wonders whether it is enough to simply assume that the Brauer-Manin obstruction to the Hasse principle is the only one for rational points on $X$. This leads one to ask when sufficiency of the Brauer-Manin obstruction can be transferred over field extensions, which is the topic of Theorem~\ref{thm:transfer}, proved below. Similar ideas have appeared in  \cite{riman}.

\begin{proof}[Proof of Theorem~\ref{thm:transfer}] Let $\alpha \in \Br X \{2\}$ be a representative for $[\alpha] \in \Br X/ \Br_0 X$. Let $\calX$ be an integral model for $X$ over $\Spec \calO_{S}$, where $S \subset \Omega_k$ is some finite subset containing the archimedean places and such that $\alpha$ comes from an element in $\Br\calX\{2\}$. 
Since $\Br \calO_v = 0$ for all $v \not \in S$, we have $\langle x_v, \alpha \rangle_{\textrm{BM}} = 0$ for any $x_v \in X(k_v)$ and $v \not\in S$. 

Let $l/k$ be a finite extension of odd degree which is completely split at all places in $S$. Suppose that  $ X(\A_l)^{\alpha_l} \neq \emptyset$, where $\alpha_l:=\Res_{l/k}\alpha$.
Let $(x_w)_w \in X(\A_l)^{\alpha_l}$. Since $\inv_w (\alpha_l(x_w)) = 0$ for any place $w$ of $l$ lying above a place $v$ of $k$ with $v\notin S$, we have 
\[0= \sum_{w \in \Omega_l } \inv_w(\alpha_l(x_w)) =  \sum_{v \in S} \sum_{w|v} \inv_w(\alpha_l(x_w)).\]
Since $l/k$ is completely split at all places in $S$, for each place $w|v$ with $v \in S$ we have \mbox{$[l_w:k_v]=1$} and hence there exists some $u^{(w)}_v \in X(k_v)$ with $\inv_w(\alpha_l(x_w)) = \inv_v(\alpha(u_v^{(w)}))$.
It follows that
\begin{equation}\label{Av} \sum_{v \in S} \sum_{w|v} \inv_w(\alpha_l(x_w)) = \sum_{v \in S} \sum_{w|v}  \inv_v(\alpha(u^{(w)}_v)) = 0.\end{equation}
We now construct an adelic point $(x'_v)_v \in X(\A_k)^{\alpha}$. By assumption, $X(\A_k) \neq \emptyset$, so let $(x''_v)_v \in X(\A_k)$. Let $x'_v := x''_v$ for all $v \not\in S$. For $v \in S$, we proceed as follows. Define
\[ A_v:=  \sum_{w|v}  \inv_v(\alpha(u^{(w)}_v)).\]
If $A_v = 0$, then there exists some $w | v$ such that $\inv_v( \alpha(u_v^{(w)})) = 0$, since $\alpha$ has order $2$ in $\Br X/\Br_0 X$ and $\# \{ w| v\} = [l:k]$ is odd. 
In this case, let $x'_v := u_v^{(w)}$.
Note that, from \ref{Av},  there is an even number (possibly zero) of places $v \in S$ such that $A_v = 1/2$; let $\calV$ be the set of all such $v \in S$. If $\calV = \emptyset$, then $A_v = 0$ for all $v \in S$. This means that we have already constructed an adelic point $(x'_v)_v \in X(\A_k)$ with the property that $(x'_v)_v \in X(\A_k)^\alpha$.
So let us suppose that $\calV \neq \emptyset$ and pair the places in $\calV$, writing $\calV = \{v_1, v_1', v_2, v'_2, ..., v_r, v'_r\}$. For $i = 1, ..., r$, we do the following. Since $A_{v_i} = 1/2$ and  $A_{v'_i} = 1/2$, we can pick some  $u^w_{v_i} \in X(k_{v_i})$ and some  $u^{w'}_{v'_i} \in X(k_{v'_i})$ such that $\inv_{v_i} (\alpha(u^w_{v_i} )) = \inv_{v'_i} (\alpha(u^{w'}_{v'_i}))  = 1/2$. We then set $x'_{v_i} := u^w_{v_i}$ and $x'_{v'_i}:= u^{w'}_{v'_i}$. At the end of this process, we obtain  an adelic point $(x'_v)_v \in X(\A_k)$ which, by construction, has the property that $(x'_v)_v \in X(\A_k)^\alpha$. 

Hence, for any finite extension $l/k$ of odd degree which is completely split at all places in $S$,  we have shown that $ X(\A_l)^{\alpha_l} \neq \emptyset$ implies that $ X(\A_k)^{\alpha} \neq \emptyset$. 
The fact that there exist uncountably many such field extensions $l/k$ follows from e.g. the proof of \cite[Thm. 4.1]{MazurRubin-DioStab}. 
\end{proof}

\begin{proof}[Proof of Corollary~\ref{thm:liftsuff}]
This is a special case of Theorem~\ref{thm:transfer} in conjunction with \cite[Thm. 1.7]{CreutzViray-DegBMO}.
\end{proof}

\bibliographystyle{amsalpha}
\bibliography{bibshort}

\end{document}